\documentclass[12pt, twoside, leqno]{article}

\usepackage{amsmath,amsthm}
\usepackage{amssymb}

\usepackage{enumitem}

\usepackage[pdftex]{graphicx}

\usepackage[T1]{fontenc}

\newtheorem{theorem}{Theorem}[section]
\newtheorem{corollary}[theorem]{Corollary}
\newtheorem{lemma}[theorem]{Lemma}
\newtheorem{proposition}[theorem]{Proposition}
\newtheorem*{problem}{Problem}

\newtheorem*{mainthm}{Main Theorem}

\theoremstyle{definition}

\newtheorem{step}{Step}

\newtheorem*{xrem}{Remark}

\numberwithin{equation}{section}

\newcommand{\Set}[2]{
\left\{\begin{array}{c|c}
\displaystyle#1&\displaystyle#2
\end{array}\right\}}
\newcommand{\id}{\mathrm{id}}
\newcommand{\R}{\mathbb{R}}
\newcommand{\M}{\mathcal{M}}
\newcommand{\K}{\mathcal{K}}
\newcommand{\dGH}{d_{\mathrm{GH}}}
\newcommand{\dist}{\mathop{\mathrm{dist}}\nolimits}
\newcommand{\diam}{\mathop{\mathrm{diam}}\nolimits}
\newcommand{\h}{\text{-}}

\begin{document}


\baselineskip=17pt


\title{An Isometric Embedding of a Bounded Set in a Euclidean Space into the Gromov-Hausdorff Space}

\author{Takuma Byakuno\\
Institute of Mathematics\\ 
Graduate School of Science and Engineering, Kansai University\\
3-3-35 Suita, 564-8680 Osaka, Japan\\
E-mail: k241344@kansai-u.ac.jp}

\date{}

\maketitle


\renewcommand{\thefootnote}{}



\renewcommand{\thefootnote}{\arabic{footnote}}
\setcounter{footnote}{0}


\begin{abstract}
We construct an isometric embedding of a bounded set in a Euclidean space into the Gromov-Hausdorff space.
In particular, we can embed a bounded and connected Riemannian manifold into the Gromov-Hausdorff space by a bilipschitz map.
\end{abstract}

\section{Introduction}
In 1909, Fr\'{e}chet \cite[pp.161--162]{Frechet} showed that an arbitrary separable metric space can be isometrically embedded into the set $l^\infty$, 
which consists of all real valued and bounded sequences $(x_n)_{n=1}^\infty=(x_1,x_2,\ldots,x_n,\ldots)$ and is equipped with the supremum distance:
\[
\mathrm{the\ distance\ between}\ (x_n)_{n=1}^\infty\ \mathrm{and}\ (y_n)_{n=1}^\infty:
\sup_{n=1,2,\ldots}|x_n-y_n|.
\]
Indeed, we can define the isometric embedding of a separable metric space $(X,d)$ into $l^\infty$ by:
\[
x\mapsto (d(x,x_n)-d(x_n,x_1))_{n=1}^\infty
\hspace{2em}
\mathrm{for}\ x\in X,
\]
where $\{x_1,x_2,\ldots\}$ is a dense subset of $X$.\\

In 1927, Urysohn \cite{Urysohn} constructed the separable complete metric space $U$ satisfying the following:
\begin{itemize}
\item
For an arbitrary separable metric space $X$, there exists a suitable subspace $M$ of $U$ such that $M$ is isometric to $X$,
\item
An arbitrary surjective isometry between finite subsets of $U$ can be extended to an isometry of $U$ onto itself.
\end{itemize}
In particular, all separable complete metric spaces satisfying the above properties 
are isometric to one another. This unique space is called the {\it Urysohn universal space}.\\

In 1932, Banach \cite[p.187]{Banach} showed that an arbitrary separable metric space can be isometrically embedded into the set $C([0,1])$, which consists of all continuous functions on the closed interval $[0,1]$ and is equipped with the supremum distance:
\[
\mathrm{the\ distance\ between}\ f\ \mathrm{and}\ g:
\sup_{0\leq t\leq 1}|f(t)-g(t)|.
\]
This fact was proved as an application of Fr\'{e}chet's isometric embedding and Banach-Mazur theorem \cite[p.185]{Banach}, which states that every Banach space can be isometrically embedded into $C([0,1])$. However $C([0,1])$ is not isometric to the Urysohn universal space according to Mazur-Ulam theorem, which states that all surjective isometries between normed spaces are affine.\\

From viewpoints of the above works, we can have the natural problem: what is the smallest metric space that contains an isometric copy of each metric spaces satisfying an arbitrary fixed property?\\

In 2017, Iliadis, Ivanov and Tuzhilin \cite{Tuzhilin1} showed that all finite metric spaces can be isometrically embedded into the Gromov-Hausdorff space.
The {\it Gromov-Hausdorff space} is a metric space consisting of a certain set $\M$ charactrized as follows and Gromov-Hausdorff distance $\dGH$, which is a distance function on $\M$ and defined later:
\begin{itemize}
\item
For every $X\in\M$, $X$ is a compact metric space,
\item
For an arbitrary compact metric space $X$,
we can find a suitable $Y\in\M$ such that $Y$ is isometric to $X$,
\item 
For every $X,Y\in\M$, we obtain $X=Y$ if $X$ is isometric to $Y$.
\end{itemize}
Although $\M$ is usually defined as the set of isometry classes of all compact metric spaces, we consider $\M$ as a set with these properties to avoid set-difficulty in this paper.
For instance, by means of Frechet's isometric embedding, a system consisting of suitable subsets of $l^\infty$ satisfies these properties. 
The Gromov-Hausdorff space $\M$ has a dense countable set consisting of isometric copies of all finite metric spaces endowed with rational valued distance functions and therefore $\M$ is separable. 
Moreover we can show that $\M$ is complete, but $\M$ is not isometric to the Urysohn universal space. See\cite[Section 3]{Tuzhilin1} and \cite[Main Theorem]{Tuzhilin2}.\\

According to \cite{Tuzhilin1}, there are many open problems about geometrical properties of $\M$. For instance, in this paper, we focus on:
\begin{problem}
Can we isometrically embed all compact metric spaces  into the Gromov-Hausdorff space $\M$? 
\end{problem}
\noindent Iliadis, Ivanov and Tuzhilin's result can be considered as a partial answer of this problem.\\

In this paper, we will prove the next theorem, which is generalization of their result and also considered as a partial answer of this problem:
\begin{mainthm}
Let $A$ be a metric subspace of the Euclidean space $\R^N\ (N=1,2,\ldots)$ equipped with the Chebyshev distance. 
If $A$ is bounded, then $A$ can be isometrically embedded into the Gromov-Hausdorff space $(\M, \dGH)$.
\end{mainthm}

\section{Preliminaries}
In this paper, we equip the Euclidean space $\R^N\ (N=1,2,\ldots)$
with the {\it Chebyshev distance} $d^N$ unless explicitly stated otherwise, which is a distance function on $\R^N$ and defined by the following:
\[
d^N((x_n)_{n=1}^N,(y_n)_{n=1}^N)=\max_{n=1,2,\ldots,N}|x_n-y_n|\hspace{2em}\mathrm{for}\ (x_n)_{n=1}^N,(y_n)_{n=1}^N\in\R^N.
\]
In particular, 
we define 
the distance between a point and a set, 
the distance between a set and a set 
and the diameter of a set in $\R^2$
as:
\begin{align*}
\mathrm{point\h set\ distance}:\ 
&&\dist(p,A)=&\inf_{x\in A}d^2(p,x),\\
\mathrm{set\h set\ distance}:\ 
&&\dist(A,B)=&\inf_{x\in A,y\in B}d^2(x,y),\\
\mathrm{diameter}:\ &&\diam(A)=&\sup_{x,y\in A}d^2(x,y)
\end{align*}
for $p\in\R^2$ and $A,B\subset\R^2$ with $A,B\neq\emptyset$, respectively.\\

For an arbitrary metric space $(Z,d_Z)$, the {\it Hausdorff distance} $Hd_Z$ of the metric space is defined by the following:
\[
Hd_Z(A,B)
=\max\left\{\sup_{x\in A}\inf_{y\in B}d_Z(x,y),\sup_{y\in B}\inf_{x\in A}d_Z(x,y)\right\}
\]
for $A,B\subset Z$\ with $A,B\neq\emptyset$.
Then we can define the {\it Gromov-Hausdorff distance} $\dGH$ by the following:
\[
\dGH(X,Y)=\inf_{(Z,\varphi,\psi)} Hd_Z(\varphi[X],\psi[Y])
\hspace{2em} \mathrm{for}\ X,Y\in\M,
\]
where $\varphi[X]$ and $\psi[Y]$ are the images of $X$ and $Y$ under $\varphi$ and $\psi$, respectively and the infimum is taken over all metric spaces $(Z,d_Z)$ and two isometric embeddings $\varphi:X\rightarrow Z$ and $ \psi:Y\rightarrow Z$.
\\

For an arbitrary positive number $\varepsilon>0$ and two metric spaces $(X,d_X)$ and $(Y,d_Y)$, a map $F:X\rightarrow Y$ is called an {\it $\varepsilon$-isometry} if $F$ satisfies:
\begin{itemize}
\item
{\it $\varepsilon$-isometric} : 
$|d_X(p,q)-d_Y(F(p),F(q))|\leq\varepsilon$\hspace{2em}for every $p,q\in X$,
\item 
{\it $\varepsilon$-surjective} : 
For an arbitrary $y\in Y$, we can find a suitable $x\in X$ such that $d_Y(F(x),y)\leq\varepsilon$.
\end{itemize}

We shall use the following proposition. 
See \cite[p.258, Corollary 7.3.28]{BBI}.

\begin{proposition}
\label{GHmap}
For $X,Y\in\M$ and $\varepsilon>0$ with $\dGH(X,Y)<\varepsilon$, 
there exist a positive number $\delta<\varepsilon$ and a $2\delta$-isometry $F:X\rightarrow Y$.
\end{proposition}

\section{Proof of Main Theorem}
In this section, we shall prove Main Theorem and show some corollaries.

\begin{proof}[Proof]
Let $A$ be a bounded subspace of $\R^N$. 
In particular, we may assume there exists $M>0$ such that
\begin{equation}
\label{assumption}
0\leq x_n\leq M\hspace{2em}\mathrm{for\ every}\ n=1,2,\ldots,N\ \mathrm{and}\ (x_n)_{n=1}^N\in A
\end{equation}
by using a translation, which preserves distance. 
Under this assumption,
\begin{align*}
d^N((x_n)_{n=1}^N,(y_n)_{n=1}^N)
=&\max_{n=1,2,\ldots,N}|x_n-y_n|
=\max_{n=1,2,\ldots,N}\max\{x_n-y_n,y_n-x_n\}\\
\leq&\max_{n=1,2,\ldots,N}\max\{M-0,M-0\}
=M.
\end{align*}

First we construct an isometric embedding of $(A,d^N)$ into $(\M,\dGH)$.
Put $C=4M$, $D=C+2M+C$ and make an arbitrary point $x=(x_n)_{n=1}^N$ fixed.
Then we define $\K_x$ by the following subspace of $\R^2$:
\[
\K_x=\bigcup_{n=1}^N\{p_n^+(x)\}\cup\{p_n^-(x)\}\cup\square(n),
\]
where
\begin{align*}
p_n^\pm(x)=&(D(n-1),\pm x_n)\ (\mathrm{double\ sign\ in\ same\ order}),\\
\square(n)=&\Set{(a+(C+D(n-1)),b)}{a\in[0,2M],b\in[-M,M]}
\end{align*}
for $n=1,2,\ldots,N$.\\

\begin{figure}[h]
\centering
\includegraphics[width=0.9\textwidth]{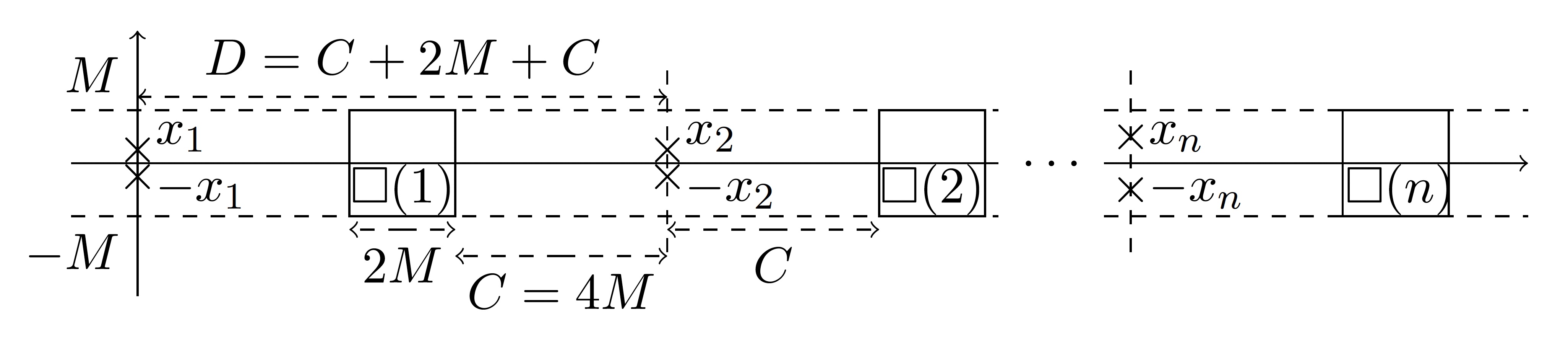}
\caption{The definition of $\K_x$.}
\end{figure}

\begin{figure}[h]
\centering
\includegraphics[width=0.9\textwidth]{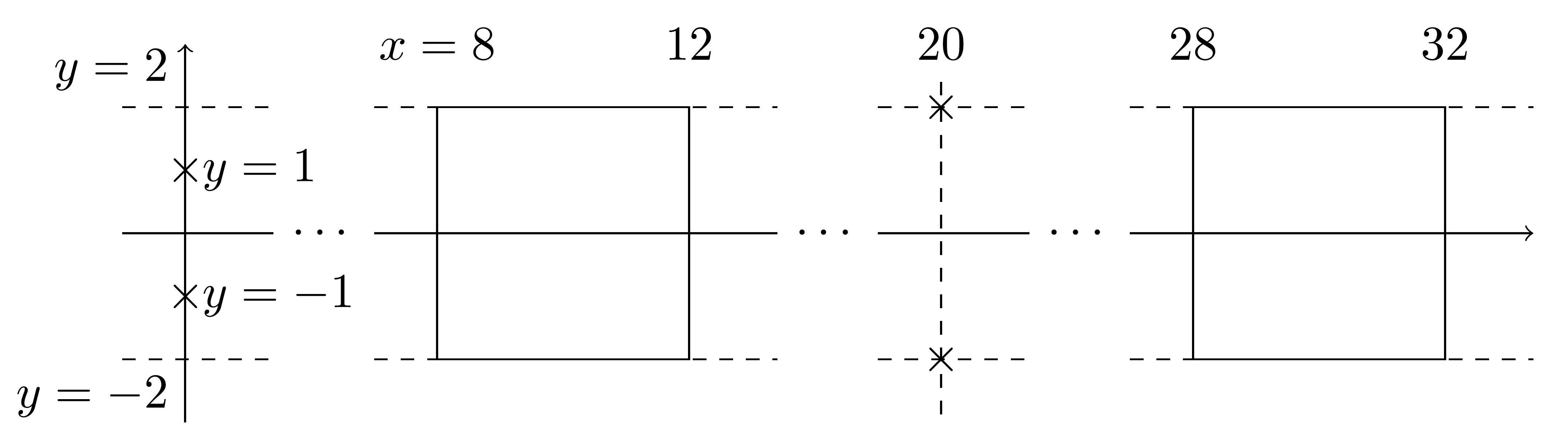}
\caption{In case of $N=2$, $M=2$, $x=(1,2)$}
\end{figure}

We have $\diam(\square(n))=2M$ for every $n=1,2,\ldots,N$ by direct calculation. 
In particular, two arbitrary vertices $P,Q$ of the square recognize its diameter: $d^2(P,Q)=2M$.
We shall use this fact to show that a constructed map preserves distance.\\

The subspace $(\K_x,d^2)$ of $\R^2$ is compact.
Indeed, all one-point sets and all squares are compact in $\R^2$ and $\K_x$ is a finite union of them.
Consequently we can choose $K(x)\in\M$ which is isometric to $\K_x$.
We consider $K(x)=\K_x$ for simplicity and denote the correspondence $x\mapsto K(x)$ by $K:A\rightarrow\M$.\\

We show that the map $K$ preserves distance.
Let $x=(x_n)_{n=1}^N$ and $y=(y_n)_{n=1}^N$ be two arbitrary fixed points in $A$.
It is necessary to show two inequalities
\[
\dGH(K(x),K(y))\leq d^N(x,y)
\]
and
\[
d^N(x,y)\leq\dGH(K(x),K(y)).
\]

We have $\dGH(K(x),K(y))\leq Hd^2(K(x),K(y))$ by the definition of $\dGH$.
Thus it is enough for the former to show $Hd^2(K(x),K(y))\leq d^N(x,y)$.
For an arbitrary point $p\in K(x)$, we shall find a suitable point $q\in K(y)$ such that $d^2(p,q)\leq d^N(x,y)$.
Since all squares $\square(n)$ is contained in $K(x)$ and $K(y)$, 
we can choose $q=p\in K(y)$ if $p$ belongs to a square $\square(n)$, which satisfies $d^2(p,q)=0\leq d^N(x,y)$.
If $p=p_n^\pm(x)\in K(x)$ for $n=1,2,\ldots,N$, 
then we can choose $q=p_n^\pm(y)\in K(y)$ (double sign in same order), and we have
\begin{align*}
d^2(p,q)
=&\max\{|D(n-1)-D(n-1)|,|\pm x_n-(\pm y_n)|\}\\
=&\max\{0,|x_n-y_n|\}
=| x_n-y_n|\leq d^N(x,y).
\end{align*}
Consequently, we obtain
\[
\sup_{p\in K(x)}\inf_{q\in K(y)}d^2(p,q)\leq d^N(x,y)
\]
as well as
\[
\sup_{q\in K(y)}\inf_{p\in K(x)}d^2(p,q)\leq d^N(x,y).
\]
It follows that $Hd^2(K(x),K(y))\leq d^N(x,y)$ by the definition of $Hd^2$.\\

For the latter inequality, suppose the contrary:
\[
\dGH(K(x),K(y))<d^N(x,y).
\]
Then we can choose a positive number $\delta<d^N(x,y)\leq M$ and a $2\delta$-isometry $F:K(x)\rightarrow K(y)$ according to Proposition \ref{GHmap}.\\

It is straightforward to see the following:
\begin{lemma}
\label{calculation}
We have the following for every $z\in A$, $P\in\square(n)$, $Q\in\square(m)$ and $s,t\in\{+,-\}$.
\begin{itemize}
\item[\rm{(1-1)}]
$d^2(p_n^+(z),p_n^-(z))\leq 2M$.
\item[\rm{(1-2)}]
$d^2(p_n^s(z),p_m^t(z))=D|n-m|$ if $n\neq m$.
\item[\rm{(2-1)}]
$\dist(p_n^\pm(z), \square(m))=C+D(m-n)$ if $n\leq m$.
\item[\rm{(2-2)}]
$\dist(p_n^\pm(z), \square(m))=D(n-m)-C-2M$ if $m+1\leq n$.
\item[\rm{(3)}]
$\dist(\square(n),\square(m))=D|n-m|-2M$ if $n\neq m$.
\item[\rm{(4)}]
$d^2(P,Q)\leq D|n-m|+2M$.
\item[\rm{(5)}]
$d^2(p_n^\pm(z),Q)\leq D-C$ if $m=n-1$.
\end{itemize}
\end{lemma}

We divide the argument among nine steps to obtain a contradiction.

\begin{step}
\label{step1}
For an arbitrary number $n=1,2,\ldots,N$, we shall prove that there exists a number $m=1,2,\ldots,N$ such that $F(\square(n))\subset\square(m)$.\\

Put $\square=\square(n)$. Firstly, for each point $P\in\square$, we show that there exists a number $m$ such that $F(P)\in\square(m)$ by contradiction.
Let $P\in\square$ be a point with $F(P)=p_m^\pm(y)$ $(m=1,2,\ldots,N)$, and suppose that we can find a point $Q\in\square$ with $F(Q)\in\square(k)$ $(k=1,2,\ldots,N)$ for the sake of contradiction.\\

We have $d^2(P,Q)\leq\diam(\square)=2M$ since $P,Q\in\square$, 
and
\begin{align*}
&d^2(F(P),F(Q))
\geq\dist(p_m^\pm(y),\square(k))\\
&\geq \begin{cases}
C+D(k-m)&\mathrm{if:} m\leq k,\\
D(m-k)-C-2M&\mathrm{if:} k+1\leq m
\end{cases}
\geq \begin{cases}
C+0&\mathrm{if:} m\leq k,\\
D-C-2M&\mathrm{if:} k+1\leq m
\end{cases}\\
&\geq C
\end{align*}
according to (2-1) and (2-2) of Lemma \ref{calculation}.
In particular, we obtain
\[
2M=C-2M\leq d^2(F(P),F(Q))-d^2(P,Q)\leq2\delta<2M
\]
since $F$ is $2\delta$-isometric. This is a contradiction.
Therefore we can not find a point $Q\in\square$ with $F(Q)\in\square(k)$ $(k=1,2,\ldots,N)$, i.e. for arbitrary point $Q\in\square$, there exist a number $k$ and a sign $s\in\{+,-\}$ such that $F(Q)=p_k^s(y)$.
In particular, we can show that $k=m$.
In fact, if $k\neq m$, then we have
\[
d^2(F(P),F(Q))=d^2(p_m^\pm(y), p_k^s(y))=D|m-k|\geq D\geq C
\]
according to (1-2) of Lemma \ref{calculation}.
This inequality and $d^2(P,Q)\leq2M$ lead to $2M<2M$ by a manner similar to above, which is a contradiction.
Thus, for arbitrary point $Q\in\square$, there exists a sign $s\in\{+,-\}$ such that $F(Q)=p_m^s(y)$, i.e. $F(\square)\subset\{p_m^+(y),p_m^-(y)\}$, which means that $F(\square)$ has two points at most. 
However, $\square$ has four vertices. 
It follows that there exist two vertices $Q,R\in\square$ such that $F(Q)=F(R)$.
Then we obtain
\[
2M=d^2(Q,R)\leq d^2(Q,R)-d^2(F(Q),F(R))\leq2\delta<2M
\]
since $F$ is $2\delta$-isometric. This is a contradiction.\\

Secondly, we show that two points $P,Q\in\square$ with $F(P)\in\square(m)$ and $F(Q)\in\square(k)$ $(k,m=1,2,\ldots)$ satisfy $k=m$.
Suppose the contrary: $m\neq k$. Then we have 
\[
d^2(F(P),F(Q))
\geq \dist(\square(m),\square(k))
\geq D-2M
\geq C
\]
according to (3) of Lemma \ref{calculation}. 
This inequality and $d^2(P,Q)\leq 2M$ lead to $2M<2M$ by a manner similar to above, which is a contradiction.\\

Consequently, for an arbitrary number $n$, there exists a number $m$ such that $F(\square(n))\subset\square(m)$.
\end{step}

\begin{step}
\label{step2}
As a result of step \ref{step1}, we can consider a map $\sigma:\{1,2,\ldots,N\}\rightarrow\{1,2,\ldots,N\}$ satisfying the following:
\[
F(\square(n))\subset\square(\sigma(n))
\hspace{2em}\mathrm{for\ every}\ n=1,2,\ldots,N.
\]
Here, we shall prove that the map $\sigma$ is injective,
and therefore the map $\sigma$ is bijective.\\

Suppose the contrary: there exist two distinct numbers $n$,$m$ such that $\sigma(n)=\sigma(m)$.
Then some points $P\in\square(n), Q\in\square(m)$ satisfy $F(P),F(Q)\in\square(\sigma (n))$ by the definition of $\sigma$.
Thus we have
\begin{align*}
2M
\leq& D-2M-2M
\leq\dist(\square(n),\square(m))-\diam(\square(\sigma(n)))\\
\leq& d^2(P,Q)-d^2(F(P),F(Q))
\leq2\delta
<2M
\end{align*}
by $F$ being $2\delta$-isometric and (3) of Lemma \ref{calculation}.
This is a contradiction.
\end{step}

\begin{step}
\label{step3}
For an arbitrary number $n$, we shall prove that there exists a number $m$ such that $F(p_n^\pm(x))\in\{p_m^+(y), p_m^-(y)\}$.\\

Suppose the contrary: $F(p_n^\pm(x))\in\square(m)$ $(m=1,2,\ldots,N)$, and 
take a leftmost point $P$ of $\square(n)$, which satisfies $d^2(p_n^\pm(x),P)=C$.\\

If $m=\sigma(n)$, then we have $F(P)\in\square(m)$ by the definition of $\sigma$, and therefore we have $d^2(F(p_n^\pm(x)),F(P))\leq2M$.
It follows that
\[
2M=C-2M\leq d^2(p_n^\pm(x),P)-d^2(F(p_n^\pm(x)),F(P))\leq2\delta<2M
\]
since $F$ is $2\delta$-isometric, which is a contradiction.\\

Thus $m\neq\sigma(n)$ and we have
\[
d^2(F(p_n^\pm(x)),F(P))\geq\dist(\square(m),\square(\sigma(n)))\geq D-2M=2C
\]
according to (3) of Lemma \ref{calculation}.
As a consequence, we have
\[
2M\leq2C-C\leq d^2(F(p_n^\pm(x)),F(P))-d^2(p_n^\pm(x),P)\leq2\delta<2M
\]
since $F$ is $2\delta$-isometric. This is a contradiction.\\

It follows that there exist two numbers $m,k$ and two signs $s,t\in\{+,-\}$ such that $F(p_n^+(x))=p_m^s(y)$ and $F(p_n^-(x))=p_k^t(y)$.
If $m\neq k$, then we have 
\[
d^2(p_n^+(x),p_n^-(x))\leq 2M
\]
and 
\[
d^2(F(p_n^+(x)),F(p_n^-(x)))=D|m-k|\geq C
\]
 according to (1-2) of Lemma \ref{calculation}.
These inequalities lead to $2M<2M$ by a manner similar to above, which is a contradiction.\\

Consequently, for an arbitrary number $n$, there exist a number $m$ and two signs $s,t\in\{+,-\}$ such that $F(p_n^+(x))=p_m^s(y)$ and $F(p_n^-(x))=p_m^t(y)$, i.e. $F(p_n^\pm(x))\in\{p_m^+(y), p_m^-(y)\}$.
\end{step}

\begin{step}
\label{step4}
As a result of step \ref{step3}, we can consider a map $\tau:\{1,2,\ldots,N\}\rightarrow\{1,2,\ldots,N\}$ satisfying the following:
\[
F(p_n^\pm(x))\in\{p_{\tau(n)}^+(y), p_{\tau(n)}^-(y)\}\hspace{2em}\mathrm{for\ every}\ n=1,2,\ldots,N.
\]
Here, we shall prove that the map $\tau$ is injective,
and therefore the map $\tau$ is also bijective.\\

Suppose the contrary: there exist two distinct numbers $n$,$m$ such that $\tau(n)=\tau(m)$. Then we have
\[
d^2(p_n^+(x),p_m^+(x))\geq D\geq C
\]
and
\[
d^2(F(p_n^+(x)),F(p_m^+(x)))\leq
\begin{cases}
0&\mathrm{if:}F(p_n^+(x))=F(p_m^+(x)),\\
2M&\mathrm{otherwise}
\end{cases}
\leq 2M
\]
according to (1-2) and (1-1) of Lemma \ref{calculation} and the definition of $\tau$. 
These inequalities and $F$ being $2\delta$-isometric lead to $2M<2M$ by a manner similar to above, which is a contradiction.
\end{step}

\begin{step}
\label{step5}
We shall prove:
\[
|\sigma(n)-\sigma(m)|=1\hspace{2em}\mathrm{for\ every}\ n,m=1,2,\ldots,N\ \mathrm{with} \ |n-m|=1.
\]

Suppose the contrary: there exist two numbers $n$, $m$ with $|n-m|=1$ such that 
$|\sigma(n)-\sigma(m)|\neq1$.
By $n\neq m$ and $\sigma$ being injective,
we have $|n'-m'|\geq2$, where $n'=\sigma(n)$ and $m'=\sigma(m)$.
Thus we have 
\[
\dist(\square(n'),\square(m'))=D|n'-m'|-2M\geq 2D-2M>C+D
\]
according to (3) of Lemma \ref{calculation}.
Taking some points $P\in\square(n)$, $Q\in\square(m)$, 
we obtain
\[
d^2(P,Q)\leq D|n-m|+2M=D+2M
\]
according to (4) of Lemma \ref{calculation}.
Therefore
\begin{align*}
2M
\leq&C-2M
=C+D-(D+2M)
<\dist(\square(n'),\square(m'))-d^2(P,Q)\\
\leq&d^2(F(P),F(Q))-d^2(P,Q)
\leq2\delta
<2M.
\end{align*}
This is a contradiction.
\end{step}

\begin{step}
\label{step6}
we shall prove:
\[
\tau(n)\in\{\sigma(n),\sigma(n)+1\}\hspace{2em}\mathrm{for\ every}\ n=1,2,\ldots,N.
\]

Suppose the contrary: there exists a number $n$ such that $m\notin\{k,k+1\}$, where $m=\tau(n)$ and $k=\sigma(n)$.
Taking a leftmost point $P$ of $\square(n)$, we have  $F(p_n^+(x))\in\{p_m^+(y), p_m^-(y)\}$, $F(P)\in\square(k)$ and $d^2(p_n^+(x),P)=C$ by the definition of $\sigma$, $\tau$ and $P$.
In particular,
\begin{align*}
&d^2(F(p_n^+(x)),F(P))
\geq\dist(F(p_n^+(x)),\square(k))\\
&=\begin{cases}
C+D(k-m)&\mathrm{if:}m\leq k-1,\\
D(m-k)-C-2M&\mathrm{if:}k+2\leq m
\end{cases}\\
&\geq\begin{cases}
C+D&\mathrm{if:}m\leq k-1,\\
2D-C-2M&\mathrm{if:}k+2\leq m
\end{cases}
=C+D
\end{align*}
according to (2-1) and (2-2) of Lemma \ref{calculation}.
Consequently we have
\[
2M\leq C+D-C\leq d^2(F(p_n^+(x)),F(P))-d^2(p_n^+(x),P)\leq2\delta<2M
\]
since $F$ is $2\delta$-isometric. This is a contradiction.
\end{step}

\begin{step}
\label{step7}
We shall prove that $\sigma$ and $\tau$ are identity maps.\\

Firstly, we show that $\sigma$ is the identity map $n\mapsto n$ or the map $n\mapsto N+1-n$.
Without loss of generality, we may assume $N\geq2$ by $\sigma$ being bijective.
Suppose that for an arbitrary number $n$ with $n+1\leq N$, we have 
\[
\sigma(n+1)=\sigma(n)\pm1.
\]
Giving a number $n$ with $n+2\leq N$ and $\sigma(n+2)=\sigma(n+1)\mp1$ (double sign in same order),
we have
\[
\sigma(n+2)=\sigma(n+1)\mp1=\sigma(n)\mp1\pm1=\sigma(n)\ \mathrm{(double\ sign\ in\ same\ order)}.
\]
This equality contradicts $\sigma$ being bijective.
Thus we obtain 
\[
\sigma(n+2)=\sigma(n+1)\pm1\ \mathrm{(double\ sign\ in\ same\ order)}
\]
for a number $n$ with $n+2\leq N$ by step \ref{step5}. 
It follows that
\begin{align*}
\sigma(n+1)=&\sigma(n)+1&\mathrm{if:} \sigma(2)=\sigma(1)+1&,\\
\sigma(n+1)=&\sigma(n)-1&\mathrm{if:} \sigma(2)=\sigma(1)-1&
\end{align*}
by induction.
In particular,
\begin{align*}
\sigma(n)=&n&\mathrm{if:} \sigma(2)=\sigma(1)+1&,\\
\sigma(n)=&N+1-n&\mathrm{if:} \sigma(2)=\sigma(1)-1&.
\end{align*}
We have $\sigma(2)=\sigma(1)\pm1$ by step \ref{step5}, and therefore 
$\sigma$ is  the identity map $n\mapsto n$ or the map $n\mapsto N+1-n$.\\

Secondly, we show that $\sigma$ is the identity map.
Suppose the contrary.
Then we have $N\geq2$ and $\sigma(n)=N+1-n$, and therefore we see that $\sigma(2)=N-1$ and $\tau(1)=N$ according to step \ref{step6}.
Let $P$ be a leftmost point of $\square(2)$, which satisfies $d^2(p_1^+(x),P)=C+D$.
Thus we have $F(P)\in\square(N-1)$ and $F(p_1^+(x))=p_N^\pm(y)$ by $\sigma(2)=N-1$, $\tau(1)=N$ and the definition of $\sigma$ and $\tau$.
It follows that $d^2(F(p_1^+(x)),F(P))\leq D-C$ according to (5) of Lemma \ref{calculation}.
Therefore
\begin{align*}
2M\leq&2C=C+D-(D-C)\\
\leq& d^2(p_1^+(x),P)-d^2(F(p_1^+(x)),F(P))\leq2\delta<2M
\end{align*}
since $F$ is $2\delta$-isometry. This is a contradiction.\\

Finally, we show that $\tau$ is the identity map.
Suppose the contrary: there exists a number $n_0$ such that $\tau(n_0)\neq n_0$.
We have $\sigma(n_0)=n_0$, $\sigma(n_0+1)=n_0+1$ since $\sigma$ is the identity map, and therefore we obtain $\tau(n_0)=n_0+1$ according to step \ref{step6}.
Then we see that $\tau(N-1)=N$ by induction and step \ref{step6}.
However, we have $\tau(N)=N$ by $1\leq\tau(N)\leq N$, $\sigma(N)=N$ and step \ref{step6} and therefore we have $\tau(N)=\tau(N-1)$.
This equality contradicts $\tau$ being injective.
\end{step}

\begin{step}
\label{step8}
There exists a number $k$ such that $d^N(x,y)=|x_k-y_k|$.
Here, we shall prove that $x_k<|x_k-y_k|$ and $y_k<|x_k-y_k|$.\\

For simplicity, we write $p_k^+(x)$ and $p_k^-(x)$ as $p_0$ and $p_1$, respectively.
As a result of step \ref{step7}, we have $F(p_0),F(p_1)\in\{p_k^+(y),p_k^-(y)\}$
and therefore we obtain the following by direct calculation:
\[
d^2(F(p_0),F(p_1))=2|y_k|=2y_k\hspace{2em}\mathrm{if}\ F(p_0)\neq F(p_1).
\]
as well as $d^2(p_0,p_1)=2|x_k|=2x_k$. 
Then we have
\[
2\delta\geq|d^2(p_0,p_1)-d^2(F(p_0),F(p_1))|=|2x_k-2y_k|=2|x_k-y_k|=2d^N(x,y)
\]
by the definition of $k$ and $F$ being $2\delta$-isometric.
This is a contradiction to $\delta<d^N(x,y)$.
Thus we obtain $F(p_0)=F(p_1)$.
Moreover we have
\begin{align*}
2x_k
=&d^2(p_0,p_1)
=|d^2(p_0,p_1)-d^2(F(p_0),F(p_1))|\\
\leq&2\delta
<2d^N(x,y)
=2|x_k-y_k|,
\end{align*}
i.e. $x_k<|x_k-y_k|$ by $d^2(F(p_0),F(p_1))=0$ and $F$ being $2\delta$-isometric.\\

Let $F(p_0)=F(p_1)=p_k^\pm(y)$. 
For the point $p_k^\mp(y)\in K(y)$ (double sign in same order), 
we can choose a point $r\in K(x)$ such that $d^2(F(r),p_k^\mp(y))\leq2\delta$ by $F$ being $2\delta$-surjective.
In patriclar, we obtain $F(r)\in\{p_k^+(y),p_k^-(y)\}$ by the definition of $K(y)=\K_y$.
Indeed, an arbitrary point $P\notin\{p_k^+(y),p_k^-(y)\}$ satisfies
$d(P,p_k^\mp(y))\geq C\geq2M>2\delta$.\\

As a result of step \ref{step7}, for an arbitrary point $p\in K(x)$, we obtain
\begin{align*}
F(p)\in&\square(n)&&\mathrm{if:} p\in\square(n)\ (n=1,2,\ldots,N),\\
F(p)\in&\{p_n^+(y),p_n^-(y)\}&&\mathrm{if:} p\in\{p_n^+(x),p_n^-(x)\}\ (n=1,2,\ldots,N).
\end{align*}
Namely, we see that $r\in F^{-1}[\{p_k^+(y),p_k^-(y)\}]\subset\{p_0,p_1\}$, and therefore we obtain $F(r)=p_k^\pm(y)$.
It follows that
\[
2y_k=d^2(F(r),p_k^\mp(y))\leq2\delta<2d^N(x,y)=2|x_k-y_k|,
\]
i.e. $y_k<|x_k-y_k|$.
\end{step}

\begin{step}
\label{step9}
We shall prove that step \ref{step8} leads to a contradiction.\\

Without loss of generality, we may assume $y_k<x_k$.
Then we have
\[
x_k<|x_k-y_k|=x_k-y_k
\]
by step \ref{step8},
and therefore $y_k<0$.
This is a contradiction to the assumption (\ref{assumption}), i.e. $0\leq y_k$.
\end{step}

As a consequence, $\dGH(K(x),K(y))<d^N(x,y)$ leads to a contradiction.
Therefore we have $d^N(x,y)\leq \dGH(K(x),K(y))$.
Hence, the map $K:A\rightarrow\M$ preserves distance.
\end{proof}

\begin{xrem}
In the proof of Main Theorem, we do not use a point belonging to the interior of $\square(n)$. In particular, we use only three of the vertices of $\square(n)$.
Therefore we can replace $\square(n)$ with 
\[
\square(n)=\Set{(a+(C+D(n-1)),b)}{(a,b)\in X},
\]
where $X$ is defined by each of the following:
\begin{align*}
X=&([0,2M]\times[-M,M])\setminus(]0,2M[\times]-M,M[),\\
X=&\{(0,-M),(0,M),(2M,-M),(2M,M)\},\\
X=&\{(0,-M),(0,M),(2M,0)\},
\end{align*}
and $]a,b[=\Set{x\in\R}{a<x<b}$.
\end{xrem}

The result of \cite{Tuzhilin1} can be proved as an application of Main Theorem.
\begin{corollary}
All finite metric spaces can be isometrically embedded into the Gromov-Hausdorff space $(\M,\dGH)$.
\end{corollary}
\begin{proof}
Let $X=\{x_1,x_2,\ldots,x_n\}$ be an arbitrary finite metric space.
The following map is an isometric embedding of $X$ into $\R^n$:
\[
x\mapsto (d(x,x_1)-d(x_1,x_1),d(x,x_2)-d(x_2,x_1),\ldots,d(x,x_n)-d(x_n,x_1))
\]
for $x\in X$. Since $X$ has only finite points, the image of $X$ under this map is a bounded set in $\R^n$, which is isometric to $X$. According to Main Theorem, this image is isometric to a subspace of $\M$.
Consequently, $X$ is also isometric to the subspace of $\M$.
\end{proof}

Let $(M,g)$ be an $m$-dimensional Riemannian manifold.
If $M$ is connected, then we can define the {\it Riemannian distance} $d_M$ by the following:
\[
d_M(x,y)=\inf_{\gamma:[a,b]\rightarrow M}\int_a^b\sqrt{g_{\gamma(t)}(\dot{\gamma}(t),\dot{\gamma}(t))}\ dt
\hspace{2em}\mathrm{for}\ x,y\in M,
\]
where the infimum is taken over all piecewise continuous differentiable curves $\gamma:[a,b]\rightarrow M$ with $\gamma(a)=x$ and $\gamma(b)=y$.
The function $d_M$ is a distance function on $M$ and the topology determined by $d_M$ is the same as the original topology of $M$.\\

We can consider $\R^N$ equipped with the {\it standard Riemannian metric} $\langle\cdot,\cdot\rangle$ as an $N$-dimensional Riemannian manifold, which is defined by the following:
\[
\left\langle\sum_{i=1}^Na_i \frac{\partial}{\partial x^i},\sum_{j=1}^Nb_j\frac{\partial}{\partial x^j}\right\rangle_p=\sum_{i=1}^Na_ib_i
\hspace{2em}\mathrm{for}\ (a_i)_{i=1}^N,(b_i)_{i=1}^N\in\R^N\ \mathrm{and}\ p\in\R^N,
\]
where $(x^1,x^2,\ldots,x^n):\R^N\rightarrow\R^N$ is the standard coordinate on $\R^N$.
Then $d_{\R^N}$ is determined by the following, which is called  the {\it Euclidean distance}:
\[
d_{\R^N}((x_i)_{i=1}^N,(y_i)_{i=1}^N)=\sqrt{\sum_{i=1}^N(x_i-y_i)^2}
\hspace{2em}\mathrm{for\ every}\ (x_i)_{i=1}^N,(y_i)_{i=1}^N\in\R^N.
\]

The next theorem is well-known. See \cite{Nash} .

\begin{proposition}[Nash embedding theorem]
There exist a number $k=1,2,\ldots$ and a differentiable map $F:M\rightarrow \R^{m+k}$ such that  $F$ is an isometric embedding:
\[
d_M(x,y)=d_{\R^{m+k}}(F(x),F(y))
\hspace{2em}
\mathrm{\it for\ every}\ x,y\in M.
\]
\end{proposition}

On the other hand, the identity map $\id_{\R^N}:\R^N\rightarrow\R^N$ is a bilipschitz map from the metric space $(\R^N,d_{\R^N})$ to the metric space $(\R^N,d^N)$ since
\begin{equation}
\label{bilipschitz}
d^N(x,y)\leq d_{\R^N}(x,y)\leq \sqrt{N} d^N(x,y)
\hspace{2em}
\mathrm{for\ every}\ x,y\in\R^N.
\end{equation}
Therefore we obtain the following.

\begin{corollary}
Let $(M,g)$ be an $m$-dimensional and connected Riemannian manifold and 
$d_M$ the Riemannian distance on $(M,g)$.
If $M$ satisfies $\sup_{x,y\in M}d_M(x,y)<$~$\infty$, then
$(M,d_M)$ can be embedded into the Gromov-Hausdorff space $(\M,\dGH)$
by a bilipschitz map.
\end{corollary}
\begin{proof}
According to Nash embedding theorem, we can take a number $k$ and a map $F:M\rightarrow\R^{m+k}$ which is an isometric embedding of $(M,d_M)$ into $(\R^{m+k},d_{\R^{m+k}})$.
Then we see that the image $F[M]$ is a bounded set in the metric space $(\R^{m+k},d^{m+k})$
by the inequality (\ref{bilipschitz}) and the assumption of $M$. Therefore we can obtain a map $K:F[M]\rightarrow\M$ which is a isometric embedding of the subspace $(F[M],d^{m+k})$ into the Gromov-Hausdorff space $(\M,\dGH)$ according to Main Theorem.
Since two maps $F$ and $K$ preserve distance and the inequality (\ref{bilipschitz})
lead to the following:
\[
\dGH(K(F(x)),K(F(y)))\leq d_M(x,y)\leq \sqrt{m+k}\dGH(K(F(x)),K(F(y)))
\]
for every $x,y\in M$.
Consequently, the composite $K\circ F:M\rightarrow\M$ is bilipschitz.
\end{proof}
\section{Acknowledgments}
\*\indent This research was financially supported by the Kansai University Grant-in-Aid for progress of research in graduate course,2024.
The author acknowledges this support.
Dr.Yoshito Ishiki, Professor Toshihiro Shoda and Supervisor Atsushi Fujioka reviewed his paper and gave him encouragement.
Yoshito Ishiki especially gave him advice about some improvements on the definition of $\K_x$ and future challenges.
The author thanks them for their kindness.

\end{document}